\documentclass[12pt]{amsart}
\usepackage{amsfonts, amssymb, amsmath, amsthm, euscript, color}
\usepackage{geometry}
\usepackage{graphicx}
\usepackage{amssymb}
\usepackage{epstopdf}
\usepackage{url}
\usepackage[all]{xy}

\swapnumbers

\theoremstyle{plain}
\newtheorem{thm}{Theorem}[section]
\newtheorem{lem}[thm]{Lemma}
\newtheorem{prop}[thm]{Proposition}

\theoremstyle{definition}
\newtheorem*{Ack}{Acknowledgement}
\newtheorem{deff}[thm]{Definition}
\newtheorem{rem}[thm]{Remark}

\newtheorem{WTthm}[thm]{Taft-Wilson Theorem}

\newtheorem{QuantumBC}[thm]{Quantum binomial coefficients}

\theoremstyle{remark}
\newtheorem{note}[thm]{}

\def\C{\mathbb{C}}

\def\P{\mathcal{P}}
\def\U{\mathcal{U}}

\newcommand*{\e}{\ensuremath{\varepsilon}}

\def\dim{\operatorname{dim}}
\def\chara{\operatorname{char}}

\def\f{\frac}
\def\ds{\displaystyle}

\begin{document}
\thispagestyle{empty}

\title[Pointed Hopf algebras of dimension $p^2$]{Classification of pointed Hopf algebras of dimension $p^2$ over any algebraically closed field}

\author{Linhong Wang}
\author{Xingting Wang*}

\address{Department of Mathematics\\
  Southeastern Louisiana University\\
  Hammond, LA 70402}
\email{lwang@selu.edu}

\address{Department of Mathematics\\
University of Washington\\
Seattle, WA 98195}
\email{xingting@uw.edu}

\thanks{* corresponding author, partially supported by U.~S.~National Science Foundation. Research of the first author supported by the Louisiana BoR [LEQSF(2012-15)-RD-A-20].}



\begin{abstract}
Let $p$ be a prime. We complete the classification of pointed Hopf algebras of dimension $p^2$ over an algebraically closed field $k$. When $\chara k \neq p$, our result is the same as the well-known result for $\chara k=0$. When $\chara k=p$, we obtain 14 types of pointed Hopf algebras of dimension $p^2$, including a unique noncommutative and noncocommutative type.
\end{abstract}

\maketitle

\section{Introduction}
\thispagestyle{empty}
In this paper, we let $p$ be a prime number, $k$ an algebraically closed field. Any vector spaces, algebras, coalgebras, Hopf algebras, linear maps, and tensor products are over $k$. Also, let $C_{p^n}$ be the cyclic group of order $p^n$ for a positive integer $n$. Our aim is to classify pointed Hopf algebras of dimension $p^2$ over $k$ of positive characteristic.

\begin{note}
Most studies on classification of finite-dimensional Hopf algebras are over an algebraically closed field of characteristic zero, concerning various cases including dimensions $p^n$, $pq$, $p^2q$ for a prime $q\neq p$.  In general, Hopf algebras of dimension $3p$ are not classified. Other than dimension 24, all Hopf algebras over $\C$ of dimension less than 32 are classified; see  \cite{andruskiewitsch2000finite, BeattieSurvey2009, BeaGar2011, BeaGar2012, ChengNg2011, ng2008}. The chronology relevant to our present purpose can be summarized as follows, assuming $\chara k=0$. (1) In 1994, Zhu \cite{zhu1994hopf} proved the Kaplansky's conjecture, that is, any $p$-dimensional Hopf algebra is isomorphic to the group algebra $k[C_p]$. (2) In 1996, Masuoka \cite{masuoka1996pn} proved that any semisimple $p^2$-dimensional Hopf algebra must be isomorphic to the group algebras $k[C_{p^2}]$ or $k[C_p\times C_p]$. (3) In 2002, using results from \cite{andruskiewitsch1998hopf} and \cite{stefan1997hopf}, Ng \cite{ng2002non} gave an affirmative answer to the question raised by Susan Montgomery asking whether the Taft algebras are the only non-semisimple Hopf algebras of dimension $p^2$. (4) In 2008, Scherotzke \cite{SchPrankOne} classified the finite-dimensional pointed rank one Hopf algebras for $\chara k=p$. A pointed \emph{rank one} Hopf algebra $H$, introduced in \cite{KropRadford}, is generated as an algebra by the first term $H_1$ of its coradical filtration, and $\dim \,k\otimes_{H_0}H_1=2$, where $k$ is regarded as the trivial right $H_0$-module and $H_1$ is regarded as a left $H_0$-module by multiplication in $H$.
\end{note}

\begin{note}
When $\chara k\neq p$, we prove in Theorem \ref{neqThm} that any $p^2$-dimensional pointed Hopf algebra is isomorphic to either a group algebra or a Taft algebra. This can be seen as a generalization of the results for $\chara k=0$ obtained in \cite{masuoka1996pn} and \cite{ng2002non}. In conclusion, there are $p+1$ types of such pointed Hopf algebras. \par

When $\chara k=p$, the classification of connected Hopf algebras of dimension $p^2$ is known from a previous work \cite{Wan} by the second author. In fact, there are eight types A1-A8 of such connected Hopf algebras; see Proposition \ref{connected}. In this paper, we prove that any non-connected pointed Hopf algebra must be isomorphic to a group algebra $k[C_{p^2}]$ or $k[C_p\times C_p]$, or a Hopf algebra of type B1-B4 as in Theorem \ref{Thethm}, that is,
\begin{itemize}
\item[(B1)] $k[ g,\, x ]\, / \, \left( g^p-1,\, x^p\right)$,
\item[(B2)] $k[ g,\, x ]\, / \, \left( g^p-1,\, x^p-x\right)$,
\end{itemize}
with the Hopf algebra structure determined by
\begin{align*}
\Delta(g)&=g\otimes g,   \hspace{.75in}  \epsilon(g)=1, \quad  S(g)=g^{-1},\\
\Delta(x)&=x\otimes 1 + 1\otimes x, \quad  \epsilon(x)=0, \quad S(x)=-x,
\end{align*}
\begin{itemize}
\item[(B3)] $k[ g,\, x ]\, / \, \left( g^p-1,\,  x^p\right)$,
\item[(B4)] $k\langle g,\, x \rangle\, / \, \left( g^p-1,\, gx-xg-g(g-1),\, x^p-x\right)$,
\end{itemize}
with the Hopf algebra structure determined by
\begin{align*}
 \Delta(g)&=g\otimes g,   \hspace{.75in}  \epsilon(g)=1, \quad S(g)=g^{-1},\\
\Delta(x)&=x\otimes 1 + g\otimes x, \quad  \epsilon(x)=0,  \quad S(x)=-g^{-1}x.
\end{align*}
Therefore, when $\chara k=p$, there are $14$ types of such pointed Hopf algebras. Among them, the type B4 is the only one which is noncommutative and noncocommutative. This Hopf algebra appears in the work of Radford as an example of pointed Hopf algebra of dimension $p^2$ with antipode of order $2p$; see \cite{Rad1977}. Moreover, Hopf algebras of B3 or B4 type also appear in the work of Scherotzke as examples of pointed rank one Hopf algebras of the second type \cite[Definition 3.8]{SchPrankOne} and the third type \cite[Example 4.4 and 4.5]{SchPrankOne}, respectively.
\end{note}

There are non-pointed examples of Hopf algebras of dimension $p^3$ when $\chara k=0$, e.g., the semisimple Hopf algebra of dimension $8$ in \cite{Mas}, which is noncommutative and noncocommutative. For Hopf algebras of dimension $p$ or $p^2$, we ask the following question parallel to the one asked by Susan Montgomery in the characteristic zero case. This question is a generalized Kaplansky's conjecture for the case of positive characteristic.

\begin{note}
\noindent \textbf{Question:} Is it true that any Hopf algebras of dimension $p$ or $p^2$ over an algebraically closed field $k$ must be pointed?
\end{note}

\begin{Ack}
We would like to acknowledge Professor Quanshui Wu and Professor James Zhang for their valuable suggestions made while the authors were attending the 2012 BIRS workshop, New Trends in Noncommutative Algebra and Algebraic Geometry, at Banff, Canada. We are also thankful to Cris Negron and Guangbin Zhuang for helpful conversations. We would like to express our gratitude to the referee for a very careful reading of the first version of the paper, and to Professor Siu-Hung Ng for pointing out the reference \cite{Rad1977} to us.
\end{Ack}

\section{Preliminary}
We use the standard notation $(H,\,m,\,u,\,\Delta,\,\e,\,S)$ to denote a Hopf algebra $H$ over $k$, where $m: H\otimes H \to H$ is the multiplication map, $u: k\to H$ is the unit map, $\Delta: H\to H\otimes H$ is the comultiplication map, $\e: H\to k$ is the counit map, and $S: H\to H$ is the antipode. We first recall some basic definitions and facts regarding $H$.

\begin{deff}\cite[Definitions 5.1.5, 5.2.1]{Mon}
The \emph{coradical} $H_0$ of $H$ is the sum of all simple subcoalgebras of $H$. The Hopf algebra $H$ is \emph{pointed} if every simple subcoalgebra is one-dimensional of the form $kg$ where $g$ is a group-like element, and $H$ is \emph{connected} if $H_0$ is one-dimensional. For each $n\geq 1$, set  \[H_n=\Delta^{-1}(H\otimes H_{n-1} + H_0\otimes H).\] The chain of subcoalgebra $H_0\subseteq H_1 \subseteq \ldots \subseteq H_{n-1}\subseteq H_n \subseteq\ldots $ is the \emph{coradical filtration} of $H$, and we have \[H=\bigcup_{n\geq 0}H_n\quad \text{and} \quad \Delta(H_n)\subseteq \sum_{i=0}^n\, H_i\otimes H_{n-i}.\] If $H$ is pointed, then the coradical $H_0$ is the group algebra $kG$ generated by group-like elements, which is a Hopf subalgebra of $H$.\end{deff}

\begin{deff}\cite[Definitions 1.3.4]{Mon}
Let $g$ and $h$ be group-like elements in $H$. Let \[\P_{g,\,h}(H)=\{x\in H\, |\; \Delta(x)=x\otimes g + h\otimes x\},\] the \emph{subspace of $(g,\,h)$-primitive elements} in $H$. In particular, $\P_{1,\,1}(H)$ is the \emph{space of primitive elements} in $H$. We will only write $\P_{g,\,h}$ if $H$ is clear from the context.
\end{deff}

\begin{WTthm} \cite[Theorem 5.4.1]{Mon}\label{WTthm}
If $H$ is a pointed Hopf algebra with the set of group-like elements $G$, then \[H_1=H_0\oplus \Big(\oplus_{g,\,h\in G}\, \P'_{g,\,h}\Big)\] where $\P'_{g,\,h}$ is a subspace of $\P_{g,\,h}$ such that
$\P_{g,\,h}=\P'_{g,\,h} \oplus k(g-h)$.
\end{WTthm}

\begin{QuantumBC}\label{quanBino}
Suppose $H$ is a Hopf algebra with $\chara k\neq p$. Let $A$ and $B$ be two elements in $H$ such that $BA=\omega AB$, where $\omega$ is a primitive $p\,$th root of unity. Then we can expand $(A + B)^p$ by using the quantum binomial coefficients, that is,  \[(A + B)^p=\sum_{i=0}^p\, \binom{p}{i}_{\omega}\, A^{p-i}B^i,\]
where $\ds{\binom{p}{i}_{\omega}=\f{(p\,!)_{\omega}}{((p-i)\, !)_{\omega}(i\,!)_{\omega}}}$ with $\ds{(j\, !)_{\omega}=\prod_{i=1}^j\, \f{1-\omega^i}{1-\omega}}$ and $(0\, !)_{\omega}=1$.
It is clear that all $\ds{\binom{p}{i}_{\omega}=0}$ except  $\ds{\binom{p}{0}_{\omega}=\binom{p}{p}_{\omega}=1}$. Therefore
$(A+B)^p =A^p+B^p$.
\end{QuantumBC}

\begin{note} \textbf{Taft algebras.} \cite{Taf}\label{Taftalgebra}
Assume that $\chara k \neq p$. Let $\omega$ be a $p\,$th primitive root of unity. The Taft algebra $T_{p,\, \omega}$ is defined as
\[k\langle g,\, x \rangle\, / \, \left( g^p -1,\, x^p,\, gx-\omega xg \right)\]
with the Hopf algebra structure determined by
\begin{align*}
\Delta(g)&=g\otimes g,      \hspace{.75in}  \epsilon(g)=1, \quad S(g)=g^{-1},\\
\Delta(x)&=x\otimes 1 + g\otimes x, \quad  \epsilon(x)=0,  \quad S(x)=-g^{-1}x.
\end{align*}
It is well known that $T_{p,\, \omega}$ is $p^2$-dimensional and pointed.
\end{note}

By Nichols and Zoeller's theorem \cite{NicZoe}, a finite-dimensional Hopf algebra $H$ is free over any Hopf subalgebra $K$. Hence $\dim K$ divides $\dim H$. This argument appears in the proof of the next lemma and in the paragraph preceding Theorem \ref{neqThm}.

\begin{lem}\label{DimPrimSpace}
Let $H$ be a Hopf algebra of dimension $p^n$ for a positive integer $n$. \emph{(i)} If $\chara k\neq p$, then $\P_{1,\,1} $ must be a zero space. \emph{(ii)} When $\chara k=p$, $0\leq \dim \P_{1,\,1}  \leq n$. If we further assume that $H$ is not connected, then $\dim \P_{1,\,1}  < n$.
\end{lem}
\begin{proof}
Suppose $\chara k=0$. If $\P_{1,\,1} $ is nonzero, then $H$ contains an infinite-dimensional Hopf subalgebra (i.e., the universal enveloping algebra $\U(\P_{1,\,1} )$), contradiction. Hence $\P_{1,\,1} =0$. Now suppose $\chara  k=q\neq 0$ for some prime $q$. Then, the dimension of the restricted enveloping algebra $\mathfrak{u}(\P_{1,\,1} )$ divides the dimension of $H$ (i.e., $q^m\, |\, p^n$), where $m=\dim \P_{1,\,1}$. Therefore $\P_{1,\,1} =0$ if $q\neq p$. On the other hand, if $q=p$ we have $0\leq \dim \P_{1,\,1}  \leq n$. Note that $\mathfrak{u}(\P_{1,\,1} )$ is connected. This completes the proof.
\end{proof}

\begin{rem}\label{nonconnected}
It follows from the preceding lemma that a Hopf algebra of dimension $p^n$ with $\chara k\neq p$ can not be connected unless it is of dimension one.
\end{rem}

\begin{lem}\label{ConjugationScalar}
Let $H$ be a finite-dimensional pointed Hopf algebra such that the set of the group-like elements $G$ is an abelian $p$-group. Suppose $g, h\in G$. Then $kG$ acts on the $(g, h)$-primitive space $\P_{g,\,h} $ by conjugation and $k(g-h)$ is a $kG$-submodule of $\P_{g,\,h} $. Moreover,
\begin{itemize}
\item[(i) ] If $\chara k\neq p$, then there is some $x\in \P_{g,\,h} $ and a group character $\chi_x: G \rightarrow k^{\times}$ such that $fxf^{-1}=\chi_x(f) x$ for any $f\in G$.
\item[(ii)] If $\chara k=p$, then there is some $x\in \P_{g,\,h} $ and a group homomorphism $\rho_x: G \rightarrow k^{+}$ such that $fxf^{-1}-x=\rho_x(f)(g-h)$ for any $f\in G$.
\end{itemize}
Therefore, if $\P_{g,\,h} /k(g-h)\neq 0$, there always exists $x\in H\setminus H_0$ satisfying \emph{(i)} or \emph{(ii)}.
\end{lem}
\begin{proof}
For any $x\in \P_{g,\,h}$, direct computation
\begin{align*}
\Delta(fxf^{-1})&=\Big(f\otimes f\Big) \Big(x\otimes g + h\otimes x\Big)  \Big(f^{-1}\otimes f^{-1}\Big) \\
                   &=fxf^{-1}\otimes fgf^{-1} + fhf^{-1}\otimes fxf^{-1}\\
                   &=fxf^{-1}\otimes g + h\otimes fxf^{-1}
\end{align*}
shows that $fxf^{-1}\in \P_{g,\,h} $. Hence $kG$ acts on $\P_{g,\,h} $ by conjugation. It is clear that $k(g-h)$ is also a $kG$-module by conjugation and $g-h\in \P_{g,\,h}$. Hence $k(g-h)$ is a simple $kG$-submodule of $\P_{g,\,h} $.

(i) $\chara k \neq p$. Then $\chara k\, \nmid\, |G|$, and so $kG$ is commutative semisimple. There exists an eigenvector $x\in \P_{g,\,h} $ such that $fxf^{-1}=\chi_x (f)x$, for some mapping $\chi_x: G\rightarrow k$. Note that $\chi_x (f_1 f_2)=\chi_x (f_1)\chi_x (f_2)$ for any $f_1, f_2 \in G$ and that $fxf^{-1}=0$ would imply $x=0$ if $\chi_x (f)=0$. Hence $\chi_x$ is a multiplicative character on $G$.

(ii) $\chara k=p$.
Denote by $I$ the augmentation ideal of $kG$, that is, the kernel of the augmentation $\e: kG\to k$, $\e(g) = 1$ for all $g \in G$.
Note that $I$ is the unique maximal ideal of $kG$.
Hence any simple $kG$-module is isomorphic to $kG/I$ and one-dimensional.
If $\P_{g,\,h} /k(g-h)$ is the zero space, we choose $x=g-h$ and $\rho_x=0$.
In the rest of the proof, assume that $\P_{g,\,h} /k(g-h)\ne 0$. We can choose $x\in \P_{g,\,h}$ such that the $kG$-submodule of $\P_{g,\,h} /k(g-h)$ generated by $\bar{x}$ is simple and then is isomorphic to $kG/I$.
Hence $Ix \subseteq k(g-h)$, that is, $fxf^{-1}-x =\rho_x (f) (g-h)$ for some mapping $\rho_x: G\rightarrow k$.
To see that $\rho_x$ is an additive character on $G$, we take $f_1, f_2\in G$.
Then we have
\begin{equation} f_2xf_2^{-1}-x=\rho_x(f_2)(g-h)\quad \text{and}
\end{equation}
\begin{equation}
(f_2f_1)x(f_2f_1)^{-1} - x=\rho_x(f_2f_1)(g-h).
\end{equation}
On the other hand, it follows from $f_1xf_1^{-1}-x=\rho_x(f_1)(g-h)$ that
\begin{equation}
(f_2f_1)x(f_2f_1)^{-1} - f_2xf_2^{-1}=\rho_x(f_1)(g-h).
\end{equation}
Therefore, comparing the equations (1)+(3) with (2), we have
\[\Big[\rho_x(f_1f_2)-\rho_x(f_2)-\rho_x(f_1)\Big](g-h)=0.\]
If $g\neq h$, then $\rho_x(f_1f_2)=\rho_x(f_1)+\rho_x(f_2)$. If $g=h$, then $\rho_x$ can be any map.
\end{proof}

\section{Main results}
Throughout this section, $H$ is a pointed Hopf algebra with $\dim H =p^2$. We proceed by the dimension of the coradical $H_0$. Note that $\dim H_0$ divides $p^2$. Hence we have $\dim H_0= 1, p,\, \text{or}\, p^2$. In particular, if $\dim H_0=p^2$ then $H=H_0=kG$, where $G\cong C_p\times C_p$ or $C_{p^2}$. The following theorem provides the classification of $H$ when the characteristic is zero or a prime not equal to $p$.

\begin{thm}\label{neqThm}
If $\chara k \neq p$, then $H$ is isomorphic to a Taft algebra $T_{p,\,\omega}$ or a group algebra $kG$ where $G\cong C_p\times C_p$ or $C_{p^2}$.
\end{thm}

\begin{proof}
Assume $\chara k \neq p$. By Remark \ref{nonconnected}, $\dim H_0>1$. Then, it is sufficient to show, when $\dim H_0=p$, that $H$ is isomorphic to a Taft algebra $T_{p,\,\omega}$. Now suppose that $H_0=k\langle g \rangle$ for some $g^p=1$. By Theorem \ref{WTthm}, there exist integers $u$ and $v$ such that
\[\P_{g^u,\, g^v} /k(g^u-g^v) \neq 0.\]
By left multiplication of $g^{-u}$ and applying Lemma \ref{DimPrimSpace} (i), we can assume without loss of generosity that $u=0$ and $v=1$. Then, by Lemma \ref{ConjugationScalar} (i), we have $kG$ acts on $\P_{1,\,g} $ by conjugation and there is some $x\in \P_{1,\,g} \setminus H_0$ such that $gxg^{-1}=\chi_x(g)x$, where $\chi_x: G \rightarrow k^{\times}$ is a group character. Set $\omega=\chi_x(g)$. Then \[\omega^p=\chi_x(g)^p=\chi_x(g^p)=\chi_x(1)=1.\]
Thus $\omega$ is a $p\,$th root of unity. We claim that $\omega\neq 1$. Otherwise, we have $gx=xg$ and so $(G-1)H$ is a Hopf ideal of $H$. Then  $0\neq\bar{x}$ is primitive in the quotient Hopf algebra $H/(G-1)H$, a contradiction to Lemma \ref{DimPrimSpace} (i). Therefore, $\omega$ must be primitive. Set $A=x\otimes 1$ and $B=g\otimes x$. It is easy to see that $BA=\omega AB$. Then, it follows from \ref{quanBino} that
\[\Delta(x^p)=\Delta(x)^p=(x\otimes 1 + g\otimes x)^p=(x\otimes 1)^p + (g\otimes x)^p=x^p\otimes 1 + 1\otimes x^p.\]
Thus $x^p\in \P_{1,\,1}$. By Lemma \ref{DimPrimSpace} (i), we have $x^p=0$. Hence,
\[H=k\langle g,\, x \rangle\, / \, \left( g^p -1,\, x^p,\, gx-\omega xg \right).\]
It can be verified that $H$ is isomorphic to the Taft Hopf algebra $T_{p,\,\omega}$ as in \ref{Taftalgebra}. The theorem is proved.
\end{proof}

Next, we consider the case when $\chara k=p$. In \cite{Wan}, the second author has classified all \emph{connected} Hopf algebras of dimension $p^2$.

\begin{prop}\cite[Theorem 6.4]{Wan}\label{connected}
Let $\chara k=p$ and $H$ be a connected Hopf algebra of dimension $p^2$. Then $H$ is isomorphic to one of the eight types of Hopf algebras presented by generators and relations as follows.

\emph{(i)} Further assume that $\dim \P_{1,\,1}=2$.
\begin{itemize}
\item[(A1)] $k\left[x,\, y\right]/\left(x^p,\, y^p\right)$,
\item[(A2)] $k\left[x,\, y\right]/\left(x^p-x,\, y^p\right)$,
\item[(A3)] $k\left[x,\, y\right]/\left(x^p-y,\, y^p\right)$,
\item[(A4)] $k\left[x,\, y\right]/\left(x^p-x,\, y^p-y\right)$,
\item[(A5)] $k\langle x,\, y\rangle/\left([x,\,y]-y,\, x^p-x,y^p\right)$,
\end{itemize}
where $x$ and $y$ are primitive (i.e.,
$\Delta\left(x\right)=x\otimes 1+1\otimes x$ and $\Delta\left(y\right)=y\otimes 1+1\otimes y$).

\emph{(ii)} Further assume that  $\dim \P_{1,\,1}=1$.
\begin{itemize}
\item[(A6)] $k\left[x,\, y\right]/(x^p,\, y^p)$,
\item[(A7)] $k\left[x,\, y\right]/(x^p,\, y^p-x)$,
\item[(A8)] $k\left[x,\, y\right]/(x^p-x,\, y^p-y)$,
\end{itemize}
where $\Delta\left(x\right)=x\otimes 1+1\otimes x$ and $\Delta\left(y\right)=y\otimes 1+1\otimes y+\sum_{i=1}^{p-1}{p\choose i}\big/p\ x^i\otimes x^{p-i}$.
\end{prop}

The following theorem provides a complete list of pointed Hopf algebras of dimension $p^2$ with $\chara k=p$.

\begin{thm}\label{Thethm}
If $\chara k= p$, then $H$ is isomorphic to a group algebra $kG$ where $G\cong C_p\times C_p$ or $C_{p^2}$, a Hopf algebra as in Proposition \emph{\ref{connected}}, or a Hopf algebra presented by generators and relations as follows:
\begin{itemize}
\item[(B1)] $k[ g,\, x ]\, / \, \left( g^p-1,\, x^p\right)$,
\item[(B2)] $k[ g,\, x ]\, / \, \left( g^p-1,\, x^p-x\right)$,
\end{itemize}
with the Hopf algebra structure determined by
\begin{align*}
\Delta(g)&=g\otimes g,   \hspace{.75in}  \e(g)=1, \quad  S(g)=g^{-1},\\
\Delta(x)&=x\otimes 1 + 1\otimes x, \quad  \e(x)=0, \quad S(x)=-x,\; \text{or}
\end{align*}
\begin{itemize}
\item[(B3)] $k[ g,\, x ]\, / \, \left( g^p-1,\,  x^p\right)$,
\item[(B4)] $k\langle g,\, x \rangle\, / \, \left( g^p-1,\, gx-xg-g(g-1),\, x^p-x\right)$,
\end{itemize}
with the Hopf algebra structure determined by
\begin{align*}
 \Delta(g)&=g\otimes g,      \hspace{.75in}  \epsilon(g)=1, \quad S(g)=g^{-1},\\
\Delta(x)&=x\otimes 1 + g\otimes x, \quad  \epsilon(x)=0,  \quad S(x)=-g^{-1}x.
\end{align*}
\end{thm}

\begin{proof}
Assume $\chara k= p$. Note that $H$ is connected if and only if $\dim H_0=1$ and that $H=H_0$ if and only if $H$ is a group algebra of order $p^2$. Hence, it is sufficient to show, when $\dim H_0=p$, that $H$ is isomorphic to a Hopf algebra described in the theorem. Again, $H_0=k\langle g \rangle$ for some  $g^p=1$. By Theorem \ref{WTthm}, we have
\[H_0\subset H_1=H_0\oplus \Big(\oplus_{g,\,h\in G}\, \P'_{g,\,h}\Big)\] where $\P'_{g,\,h}$ is a subspace of $\P_{g,\,h}$ such that
$\P_{g,\,h}=\P'_{g,\,h} \oplus k(g-h)$.

\textbf{Case} $\P_{1,\,1}  \neq 0$.
For any $x\in \P_{1,\, 1}$, by $\Delta(x)=x\otimes1 + 1\otimes x$, we have $gxg^{-1}\in \P_{1,\, 1}$.
Moreover, by Lemma \ref{DimPrimSpace}  (ii), $\P_{1,1}$ is an one-dimensional $kG$-module by conjugation.
Every simple $kG$-module is trivial, since $G$ is a $p$-group and $\chara k=p$.
Hence, $gxg^{-1}=x$ for any $x\in \P_{1,\, 1}$.
On the other hand, we choose nonzero $x\in \P_{1,\,1}$.
Then any element in $\P_{1,\,1} $ must be in the form $\lambda x$ for some $\lambda \in k$.
Since $x\in \P_{1,\,1}$ and $\chara k=p$, we see that $\Delta(x^p)=x^p\otimes 1 + 1\otimes x^p$, and so $x^p\in \P_{1,\,1} $.
This implies that $x^p= \alpha x$ for some $\alpha \in k$.
As $\bar{k}=k$, we can rescale $x$ properly so that $x^p=0$ or $x^p=x$.
Therefore, $H$ must be a Hopf algebra generated by the group-like element $g$ and the (1,1)-primitive element $x$ subject to the relations
\[g^p-1,\, gx-xg,\, x^p-x\]
or
\[g^p-1,\, gx-xg,\, x^p.\]
It is easy to verify that the comultiplication, counit, and antipode must be as those described in (B1) and (B2).

\textbf{Case} $\P_{1,\, 1} =0$. By Theorem \ref{WTthm}, there exist integers $u$ and $v$ such that
\[\P_{g^u,\, g^v} /k(g^u-g^v) \neq 0.\]
By left multiplicating $g^{-u}$ to $x$ and reselecting the generator of $G$ if necessary, we can assume $u=0$ and $v=1$. Then, by Lemma \ref{ConjugationScalar} (ii), there exists $x\in \P_{1,\,g} \setminus H_0$ such that $gx-xg=\lambda g(g-1)$ for some $ \lambda\in k$. If $\lambda=0$, we have $g$ and $x$ are commutative, and then $\Delta(x^p)=x^p\otimes 1 + 1\otimes x^p$. Thus $x^p\in \P_{1,\,1} $, and so $x^p=0$. This gives the Hopf algebra described in (B3).
If $\lambda \neq 0$, we can assume by rescaling that $\lambda=1$, that is, $[g,x]=g(g-1)$. By \cite[Chapter V Section 7, 186--187]{Jac}, we have
\[\Delta(x^p)= (x\otimes 1 + g\otimes x)^p
=x^p\otimes 1 + 1\otimes x^p+ (\textbf{ad}\, x)^{p-1} (g)\otimes x,\] where $\textbf{ad}\, x: g \mapsto [x,g]=g(1-g)$ is a $k$-derivation on $kG$. To find  $(\textbf{ad}\, x)^{p-1} (g)$, we consider the matrix of the map $\textbf{ad}\, x$ under the basis $1,\, g,\, g^2,\, \ldots,\, g^{p-1}$ of $kG$,
{\small
\[T = \left(%
\begin{array}{cccccccc}
 0& 0 &  0 &  0 &\ldots &0 &0 &-(p-1) \\
 0 & 1 & 0 & 0 &\ldots &0&0 &0 \\
 0 & -1 & 2 & 0 &\ldots &0 &0 &0   \\
 0 & 0 & -2 & 3 &\ldots &0&0 &0   \\
 \vdots & \vdots & \vdots & \vdots  & \ddots &\vdots &\vdots &\vdots \\
 0 & 0 & 0 & 0 &\ldots &  -(p-3) &(p-2)& 0 \\
 0 & 0 & 0 & 0 & \ldots &0 & -(p-2) & (p-1) \\
\end{array}%
\right).
\]
}
The matrix $T$ can be diagonalized as
{\small
\[PTP^{-1} =
\left(
 \begin{array}{ccccc}
   0  & \ &\  &  \\
    & 1 &\  & \ \\
    &\ & 2\\
    &  & & \ddots\\
    & & & & (p-1)
 \end{array}
\right) \quad \quad \text{by}\]
}
{\small
\begin{align*}
P &= \left(%
\begin{array}{ccccccccc}
 1 &  1 & 1 & 1 &1 & \ldots & 1 &1  \\
 0 & 1 & 0 &0 &0 &\ldots &0 &0\\
 0 &-1 & 1 & 0 &0 &\ldots &0 & 0  \\
 0&1 & -2 & 1 &0 &\ldots & 0 & 0\\
 0&-1 & 3 & -3 &1 &\ldots & 0 & 0  \\
 \vdots & \vdots & \vdots & \vdots &\vdots &\ddots &\vdots &\vdots   \\
0 & 1 & -\binom{p-3}{1} &\binom{p-3}{2} & -\binom{p-3}{3}&\ldots &1 & 0\\
0& -1 & \binom{p-2}{1} &- \binom{p-2}{2}  & \binom{p-2}{3} & \ldots & -\binom{p-2}{p-3}  & 1 \\
\end{array}%
\right)\quad \text{and} \\
P^{-1}&= \left(%
\begin{array}{ccccccccc}
 1 &  -\binom{p-1}{1} & -\binom{p-1}{2} & -\binom{p-1}{3} &-\binom{p-1}{4} & \ldots & -\binom{p-1}{p-2} &-1  \\
 0 & 1 & 0 &0 &0 &\ldots &0 &0\\
 0 &1 & 1 & 0 &0 &\ldots &0 & 0  \\
 0&1 & 2 & 1 &0 &\ldots & 0 & 0\\
 0&1 & 3 & 3 &1 &\ldots & 0 & 0  \\
 \vdots & \vdots & \vdots & \vdots &\vdots &\ddots &\vdots &\vdots   \\
0 &1 & \binom{p-3}{1} &\binom{p-3}{2} & \binom{p-3}{3}&\ldots &1 & 0\\
0& 1 & \binom{p-2}{1} & \binom{p-2}{2}  & \binom{p-2}{3} & \ldots & \binom{p-2}{p-3}  & 1 \\
\end{array}%
\right).
\end{align*}
}
Since $\chara k=p$, by Fermat's Little Theorem, we have
\[\big(PTP^{-1}\big)^{p-1}=I-E_{11},\quad \text{and\; so} \quad T^{p-1}=I-P^{-1}E_{11}P,\]
where $I$ is the identity matrix and where $E_{11}$ is the elementary matrix with the (1,1) entry being 1 and zeros elsewhere.
Then, it is easy to check that
\[(\textbf{ad}\, x)^{p-1} (g) =T^{p-1}(g)=[I-P^{-1}E_{11}P] (g) =g-1.\]
Now we have
\[\Delta(x^p)=x^p\otimes 1 + 1\otimes x^p+ (\textbf{ad}\, x)^{p-1} (g) \otimes x=x^p\otimes 1 + 1\otimes x^p+ (g-1)\otimes x.\]
Then, direct computation
\begin{align*}
\Delta(x^p-x)&=\Delta(x^p)-\Delta(x)\\
                &=x^p\otimes 1 + 1\otimes x^p+ (g-1)\otimes x -x\otimes 1 - g\otimes x\\
                &=(x^p-x)\otimes 1 + 1\otimes(x^p-x)
\end{align*}
shows that $x^p-x \in \P_{1,\,1} $ and so $x^p-x=0$. Therefore, $H$ must be a Hopf algebra described in (B4).
\end{proof}


\begin{thebibliography}{99}
\bibitem{andruskiewitsch2000finite} N.~Andruskiewitsch, About finite dimensional Hopf algebras. \emph{Contemp. Math.}, 294(2002), 1–-57.
\bibitem{andruskiewitsch1998hopf} N.~Andruskiewitsch and H.-J. Schneider, Hopf algebras of order $p^2$ and braided Hopf algebras of order $p$, \emph{J.~Algebra}, 199(1998), 430--454.
\bibitem{BeattieSurvey2009} M.~Beattie, A survey of Hopf algebras of low dimension,\emph{ Acta Appl.~Math.}, 108(2009), 19–-31.
\bibitem{BeaGar2011} M.~Beattie and G.~Garcia, Techniques for classifying Hopf algebras and applications to dimension $p^3$, preprint arXiv:1108.6037.
\bibitem{BeaGar2012} \bysame, Classifying Hopf algebras of a given dimension, preprint arXiv:1206.6529.
\bibitem{ChengNg2011} Y.-L.~Cheng and S.-H.~Ng, On Hopf algebras of dimension $4p$, \emph{J.~Algebra} 328(2011), 399--419.
\bibitem{Jac}  N.~Jacobson, \emph{Lie Algebras}, Dover Publications Inc., New York, 1979.
\bibitem{KropRadford} L.~Krop and D.~Radford, Finite-dimensional Hopf algebras of rank one in characteristic zero, \emph{J.~Algebra}, 302(2006), 214--230.
\bibitem{Mas} A.~Masuoka, Semisimple Hopf algebras of dimension 6, 8, \emph{Israel J. Math.}, 92(1995), 361--373.
\bibitem{masuoka1996pn}\bysame, The $p^n$ theorem for semisimple Hopf algebras, \emph{Proc.~Amer.~Math.~Soc.}, 124(1996), 735--738.
\bibitem{Mon} S.~Montgomery, \emph{Hopf Algebras and Their Actions on Rings}, CBMS Regional Conference Series in Mathematics 82, American Mathematical Soceity, Providence, Rhode Island, 1993.
\bibitem{ng2002non} S.-H.~Ng, Non-semisimple Hopf algebras of dimension $p^2$, \emph{J.~Algebra}, 255(2002), 182--197.
\bibitem{ng2008} \bysame, Hopf algebras of dimension $pq$ II, \emph{J.~Algebra} 319(2008), 2772--2788.
\bibitem{NicZoe} W.~Nichols and M.~Zoeller, A Hopf algebra freeness theorem, \emph{Amer. J. Math.}, 111(1989), 381–-385.
\bibitem{Rad1977} D.~Radford, Operators on Hopf algebras, \emph{Amer.~J.~Math}, 99(1977), 139--158.
\bibitem{SchPrankOne} S.~Scherotzke, Classification of pointed rank one Hopf algebras, \emph{J.~Algebra}, 319(2008), 2889--2912.
\bibitem{stefan1997hopf} D.~\c{S}tefan, Hopf subalgebras of pointed Hopf algebras and applications, \emph{Proc.~Amer.~Math.~Soc.}, 125(1997), 3191--3194.
\bibitem{Taf} E.~Taft, The order of the antipode of finite-dimensional Hopf algebra, \emph{Proc.~Nat.~Acad.~Sci.} \emph{U.S.A.}, 68(1971), 2631--2633.
\bibitem{Wan} X.~Wang, Connected Hopf algebras of dimension $p^2$, \emph{J.~Algebra}, 391(2013), 93--113. doi: 10.1016/j.jalgebra.2013.06.008
\bibitem{zhu1994hopf} Y.~Zhu, Hopf algebras of prime dimension, \emph{Internat.~Math.~Res.~Notices}, 1(1994), 53--59.
\end{thebibliography}
\end{document}